\documentclass[11pt]{amsart}
\usepackage[margin=1.05in]{geometry}
\usepackage{amsmath,amssymb,mathtools}
\usepackage{tikz}
\usepackage{float}
\usetikzlibrary{arrows.meta,positioning,decorations.pathmorphing}
\definecolor{brightred}{RGB}{225,55,65}
\definecolor{brightblue}{RGB}{30,105,210}
\definecolor{brightgreen}{RGB}{30,170,95}
\definecolor{turquoise}{RGB}{15,180,185}
\definecolor{brightorange}{RGB}{245,145,30}
\definecolor{brightviolet}{RGB}{145,75,205}
\definecolor{softgray}{RGB}{238,241,245}
\tikzset{
 section/.style={brightred,very thick},
 fiber/.style={brightblue,very thick},
 dual/.style={brightgreen!85!black,very thick},
 tube/.style={brightorange,very thick},
 sumarrow/.style={-{Latex[length=2.5mm]},very thick,brightblue},
 note/.style={font=\small,align=center}
}
\usepackage[hidelinks]{hyperref}
\hypersetup{
 pdftitle={Gurtas Lefschetz Fibrations, Rational Blowdowns, and Exotic Symplectic Four-Manifolds},
 pdfauthor={Anar Akhmedov and Sumeyra Sakalli},
 pdfsubject={Lefschetz fibrations, geometric duals, and rational blowdowns},
 pdfkeywords={Lefschetz fibration, symplectic four-manifold, rational blowdown, knot surgery, geometric dual}
}
\usepackage[nameinlink,noabbrev]{cleveref}

\newtheorem{theorem}{Theorem}[section]
\newtheorem{proposition}[theorem]{Proposition}
\newtheorem{lemma}[theorem]{Lemma}
\newtheorem{corollary}[theorem]{Corollary}
\theoremstyle{remark}
\newtheorem{remark}[theorem]{Remark}

\newcommand{\CP}{\mathbb{CP}}
\newcommand{\CPb}{\overline{\mathbb{CP}}{}^{\,2}}
\newcommand{\Mod}{\operatorname{Mod}}

\newcommand{\Int}{\operatorname{int}}

\newcommand{\normal}[1]{\left\langle\!\left\langle #1\right\rangle\!\right\rangle}

\title[Gurtas fibrations and rational blowdowns]{Gurtas Lefschetz Fibrations, Rational Blowdowns, and Exotic Symplectic Four-Manifolds}
\author{Anar Akhmedov}
\address{School of Mathematics, University of Minnesota, Minneapolis, MN 55455, USA}
\email{akhmedov@umn.edu}
\author{S\"umeyra Sakall\i}
\address{Department of Mathematics and Statistics, University of South Florida, Tampa, FL 33620, USA}
\email{sumeyras@usf.edu}

\begin{document}
\begin{abstract}
We study negative spheres obtained from exceptional sections of Gurtas Lefschetz fibrations.  Starting with the known system of $4n$ disjoint $(-1)$-sections, we show that a marked double sum contains $4n$ symplectic $(-2)$-spheres and $2n$ smooth $(-4)$-spheres obtained by pairwise tubing.  By carrying the same sections through a fourfold sum, we instead obtain $4n$ disjoint symplectic $(-4)$-spheres and hence symplectic rational blowdowns.  We then consider the Lefschetz fibrations on knot-surgered elliptic surfaces.  Their branched-cover description gives $2n$ branch sections together with simultaneous geometric duals arising from node resolution.  In the knot-surgery double these sections glue to symplectic $(-4)$-spheres, while the duals make the complement of every subcollection simply connected.  The resulting rational blowdowns provide simply connected exotic symplectic four-manifolds under the parity condition stated in the main theorem.  We also record the boundary-multitwist relation determined by the Gurtas sections.
\end{abstract}
\maketitle

\section{Introduction}
Matsumoto's genus-two fibration and its higher-genus analogues have been used repeatedly in constructions of smooth four-manifolds.  Their exceptional sections provide surfaces used in many constructions of exotic smooth structures \cite{Akhmedov2008,AkhmedovPark,AMgenus2}.  Tanaka \cite{Tanaka} found boundary factorizations with $4g+4$ disjoint $(-1)$-sections in a family of hyperelliptic fibrations, and Hamada \cite{Hamada} studied the section problem for the Matsumoto-Cadavid-Korkmaz (MCK) fibrations.  Here these section systems are applied to Gurtas fibrations and their fiber sums.

Let
\begin{equation}\label{eq:Y}
 Y(n,k)\cong \Sigma_k\times S^2\#4n\CPb,
 \qquad g=2k+n-1,
\end{equation}
and let $f_{n,k}\colon Y(n,k)\to S^2$ denote the genus-$g$ Gurtas Lefschetz fibration \cite{Gurtas1,Gurtas2,AS}.  Akhmedov-Saglam \cite[Lemma~4.5]{AS} constructed $4n$ pairwise disjoint $(-1)$-sphere sections for $n\ge2$; the case $n=1$ is the even-genus MCK specialization and is covered by Hamada's four-boundary relation.

Tanaka's relation belongs to the genus-$(n-1)$ hyperelliptic summand used in the Akhmedov-Saglam construction, whereas the Gurtas fibration appears after the genus-enlarging fiber sum.  After choosing a reference fiber and framed neighborhoods of the section points (the points where the sections meet the fiber), the $4n$ disjoint $(-1)$-sections determine a factorization in $\Mod(\Sigma_g^{4n})$ with boundary term $t_{\delta_1}\cdots t_{\delta_{4n}}$.  This relation records the section points and framings; it does not identify a term-by-term stabilization of Tanaka's word with Gurtas' chosen relator.

The sections above will be used in two kinds of marked fiber sum.

There is an important distinction between the double and fourfold marked fiber sums.  Two corresponding sections join to form a symplectic $(-2)$-sphere; tubing such spheres in pairs gives smooth $(-4)$-spheres, but adjunction obstructs symplectic representatives of their classes for the specified sum form.  Four corresponding sections join directly to a symplectic $(-4)$-sphere, to which Symington's rational blowdown applies \cite{Symington}.  Rational blowdown here replaces the disk-bundle neighborhood $V_{-4}$ by the rational ball $B_2$ with the same lens-space boundary.

In this paper we give constructions via rational blowdown.  The rational blowdown problem also requires control of the fundamental group.  A dual sphere for one component kills its meridian, but its punctured disk may meet other components.  By a \emph{simultaneous} system of duals we mean that the dual to each chosen sphere meets that sphere once and is disjoint from all the others.  For the $2n$ vertical branch sections, one resolution sphere at each chosen node supplies such a system.

Here is the main result of the paper.
\begin{theorem}\label{thm:intro-main}
Let $K$ be a fibered knot of genus $h\ge1$ and $n\ge2$.  The knot-surgery double $D_K(n)$ contains $2n$ pairwise disjoint symplectic $(-4)$-spheres with simultaneous geometric duals.  The complement of every subcollection, and hence every rational blowdown along such a subcollection, is simply connected.  If $1\le m\le2n$ and $m\not\equiv0\pmod8$, the blowdown along $m$ spheres is homeomorphic but not diffeomorphic to
\[
 \#_{6n+4h-5}\CP^2\ \#\ \#_{22n+4h-5-m}\CPb.
\]
\end{theorem}

Our earlier paper \cite{AkhmedovSakalliIJM} used nodal deformations to obtain related negative configurations and generalized rational blowdowns.  Lefschetz fibrations and symplectic surgeries have led to many constructions of exotic four-manifolds; see, for example, \cite{Sakalli}.  The complement construction for one branch section is the one used by Akhmedov-Etnyre-Mark-Smith.  They showed that removing a regular fiber and a negative section from a Lefschetz fibration gives a Stein domain whose boundary contact structure is supported by the corresponding punctured-fiber open book.  Applying this construction to knot-surgery fibrations, they produced infinitely many contact three-manifolds, each admitting infinitely many homeomorphic but mutually nondiffeomorphic simply connected Stein fillings \cite[Theorem~1.1 and Lemma~4.1]{AEMS}.  Akhmedov-Ozbagci extended this picture to arbitrary finitely presentable groups.  For every finitely presentable group $G$, they constructed an isolated complex surface singularity link whose canonical contact structure admits infinitely many exotic Stein fillings with fundamental group $G$.  They also constructed infinitely many closed exotic symplectic four-manifolds with fundamental group $G$, each carrying a non-holomorphic Lefschetz fibration over $S^2$ \cite[Theorems~1 and~2]{AO}.

After recording the section relation in Section~2, we treat the double and fourfold sums in Section~3, geometric duals in Sections~4-5, and the resulting invariants and exotic structures in Section~6.

\section{Boundary relations and the Gurtas section system}
This section records the boundary relation determined by the Gurtas section system.  Throughout, $\Mod(\Sigma_g^r)$ denotes the group of orientation-preserving diffeomorphisms of the genus-$g$ surface with $r$ boundary components, fixing the boundary pointwise, modulo isotopy relative to the boundary.  All twists are right-handed Dehn twists, and products act from right to left.

\begin{proposition}\label{prop:correspondence}
Let
\[
 t_{a_N}\cdots t_{a_1}=1\in\Mod(\Sigma_g)
\]
define a Lefschetz fibration with pairwise disjoint sections $S_1,\ldots,S_r$ satisfying $S_i^2=-m_i$.  After deleting disjoint disks about the section points and choosing compatible lifts of the vanishing cycles, one has
\[
 t_{\widetilde a_N}\cdots t_{\widetilde a_1}
 =t_{\delta_1}^{m_1}\cdots t_{\delta_r}^{m_r}
 \quad\text{in }\Mod(\Sigma_g^r).
\]
Conversely, such a positive boundary-multitwist factorization determines the corresponding disjoint sections, provided the capped vanishing cycles remain homotopically nontrivial.
\end{proposition}
\begin{proof}
Remove mutually disjoint fibered tubular neighborhoods of the sections.  Over the complement of the critical values, parallel transport fixes each new boundary circle pointwise.  The obstruction to extend the induced boundary framing across the base is the relative Euler number of the section normal bundle.  Traversing the boundary of the punctured base therefore produces $-S_i^2=m_i$ right-handed twists about $\delta_i$.  Multiplying the lifted local monodromies gives the displayed relation.  Conversely, cap each boundary component by a disk: the disk centers trace disjoint sections, and the exponent of $t_{\delta_i}$ is the negative of the corresponding self-intersection.  See \cite{Tanaka,Hamada} for this standard correspondence.
\end{proof}

By a \emph{marked fiber} we mean a regular fiber together with the ordered section points and their normal framings.  A \emph{marked gluing} preserves this data.  Tanaka constructed a genus-$h$ boundary factorization encoding $4h+4$ disjoint sections.  Akhmedov-Saglam constructed the Gurtas fibration by a genus-enlarging fiber sum in which these section neighborhoods remain fixed.  Thus Tanaka's marked section data belong to the hyperelliptic summand, while the Gurtas fibration is obtained after the enlargement.

For $n\ge2$, put $h=n-1$.  Tanaka's genus-$h$ word has $4h+4=4n$ boundary twists.  In the Akhmedov-Saglam construction, the genus is enlarged at two points disjoint from these section points.  The following statement, illustrated in \cref{fig:enlargement}, records only the topology of the marked fiber; it does not transport the positive word.

\begin{lemma}\label{lem:enlargement}
Let $D_-$ and $D_+$ be disks in $\Sigma_h^{4n}$ disjoint from its distinguished boundary components.  Gluing a copy of $\Sigma_k^1$ to each of the two new boundary circles gives
\begin{equation}\label{eq:surface-enlargement}
 \Sigma_{h+2k}^{4n}\cong
 \bigl(\Sigma_h^{4n}\setminus(\Int D_-\cup\Int D_+)\bigr)
 \cup_{\partial D_-}\Sigma_k^1\cup_{\partial D_+}\Sigma_k^1,
\end{equation}
and fixes all $4n$ distinguished boundary components pointwise.
\end{lemma}
\begin{proof}
The two attachments raise the genus by $2k$ and leave the $4n$ distinguished boundary circles unchanged.  The classification of compact oriented surfaces gives the stated diffeomorphism.
\end{proof}

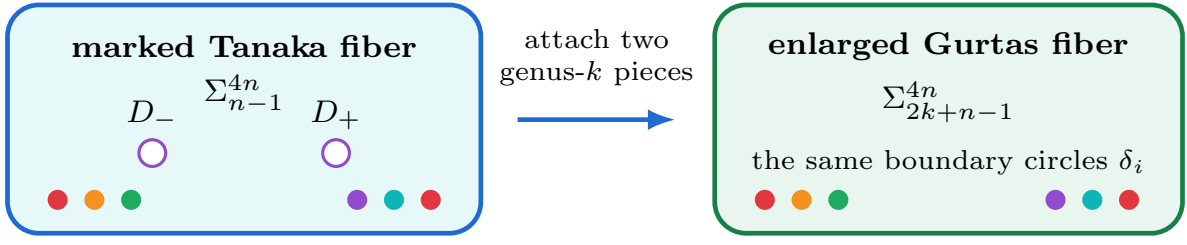
\begin{figure}[t]
\centering
\resizebox{.96\textwidth}{!}{%
\begin{tikzpicture}[>=Latex]
\filldraw[fill=turquoise!10,draw=brightblue,rounded corners=8pt,very thick] (0,0) rectangle (4.25,2.05);
\node[font=\scriptsize\bfseries] at (2.125,1.68) {marked Tanaka fiber};
\node[font=\scriptsize] at (2.125,1.25) {$\Sigma_{n-1}^{4n}$};
\foreach \x/\c in {.45/brightred,.78/brightorange,1.11/brightgreen,3.14/brightviolet,3.47/turquoise,3.80/brightred}
  \fill[\c] (\x,.30) circle (.09);
\filldraw[fill=white,draw=brightviolet,thick] (1.30,.72) circle (.12) node[above=2pt,font=\scriptsize] {$D_-$};
\filldraw[fill=white,draw=brightviolet,thick] (2.95,.72) circle (.12) node[above=2pt,font=\scriptsize] {$D_+$};
\draw[sumarrow] (4.58,1.02)--(6.0,1.02);
\node[align=center,font=\tiny] at (5.29,1.58) {attach two\\genus-$k$ pieces};
\filldraw[fill=brightgreen!10,draw=brightgreen!75!black,rounded corners=8pt,very thick] (6.35,0) rectangle (10.55,2.05);
\node[font=\scriptsize\bfseries] at (8.45,1.68) {enlarged Gurtas fiber};
\node[font=\scriptsize] at (8.45,1.18) {$\Sigma_{2k+n-1}^{4n}$};
\foreach \x/\c in {6.80/brightred,7.13/brightorange,7.46/brightgreen,9.41/brightviolet,9.74/turquoise,10.07/brightred}
  \fill[\c] (\x,.30) circle (.09);
\node[font=\tiny] at (8.45,.64) {the same boundary circles $\delta_i$};
\end{tikzpicture}}
\caption{The marked-surface enlargement in \Cref{lem:enlargement}.  The figure records the surface topology only, not a term-by-term stabilization of a monodromy word.}\label{fig:enlargement}
\end{figure}

We call $(n,k)$ \emph{admissible} when $n\ge2$, $k\ge1$, or when $n=1$, $k\ge2$, and again write $f_{n,k}\colon Y(n,k)\to S^2$ for the genus-$g$ Gurtas Lefschetz fibration.

\begin{corollary}\label{thm:boundary-lift}
For every admissible $(n,k)$, there is a monodromy factorization $t_{a_N}\cdots t_{a_1}=1$ representing $f_{n,k}$ and lifts $\widetilde a_i\subset\Sigma_g^{4n}$ such that
\begin{equation}\label{eq:boundary-lift}
 t_{\widetilde a_N}\cdots t_{\widetilde a_1}
 =t_{\delta_1}\cdots t_{\delta_{4n}}
 \quad\text{in }\Mod(\Sigma_g^{4n}).
\end{equation}
The boundary components correspond to $4n$ pairwise disjoint sections of square $-1$.
\end{corollary}
\begin{proof}
For $n\ge2$, Akhmedov-Saglam \cite[Lemma~4.5]{AS} construct $4n$ pairwise disjoint $(-1)$-sections of $f_{n,k}$.  Delete fiberwise disk neighborhoods of their intersection points with a reference fiber.  Parallel transport preserves the resulting boundary components, and \Cref{prop:correspondence} gives a relation in $\Mod(\Sigma_g^{4n})$; each boundary exponent equals $-S_i^2=1$.  Capping gives a factorization monodromy-equivalent to that of $f_{n,k}$, up to Hurwitz moves and global conjugation.  For $n=1$, Hamada's four-boundary relation for the even-genus MCK fibration gives the same conclusion \cite{Hamada}.
\end{proof}

\section{Section spheres in double and fourfold sums}
Let $E_1,\ldots,E_{4n}$ be the sections in \Cref{thm:boundary-lift}, and let $F$ be a regular fiber.  Then
\begin{equation}\label{eq:intersection-data}
 F^2=0,\qquad E_a^2=-1,\qquad E_a\cdot E_b=0\ (a\ne b),\qquad E_a\cdot F=1.
\end{equation}
We use the Akhmedov-Saglam section system: these sections arise from simultaneous symplectic exceptional spheres and may be taken symplectic for one Lefschetz-compatible form \cite[Lemma~4.5]{AS}.

We use $E_a$ for a Gurtas $(-1)$-section, $S_a$ for its $(-2)$-sphere in a double sum, and $R_a$ for its $(-4)$-sphere in a fourfold sum.  Later, $Q_a$ denotes a branch $(-2)$-section in $E(n)_K$ and $P_a$ the resulting $(-4)$-sphere in the knot-surgery double.

\subsection{The double sum}
Take two copies $Y_L,Y_R$ of $Y(n,k)$ and a marked fiber diffeomorphism $\varphi\colon F_L\to F_R$.  Form
\[
 X_\varphi=Y_L\#_{\varphi,F}Y_R.
\]
Deleting a fiber neighborhood turns each section into a disk.  The marked gluing joins the corresponding disks to a sphere
\[
 S_a=(E_a^L)^\circ\cup(E_a^R)^\circ,
 \qquad S_a^2=-2.
\]

\begin{proposition}\label{prop:minus-two}
There is a marked symplectic fiber-sum gluing for which $X_\varphi$ contains $4n$ pairwise disjoint symplectic spheres $S_1,\ldots,S_{4n}$ of square $-2$.
\end{proposition}
\begin{proof}
Choose disjoint product charts at all $4n$ marked section points, with split symplectic forms and compatible orientations.  A fiber diffeomorphism can match the ordered points and framed normal circles simultaneously.  Gompf's relative symplectic sum then glues every pair of section disks symplectically; the relative normal Euler numbers add, and distinct marked points give disjoint glued spheres \cite{Gompf}.
\end{proof}

Pair the spheres as $(S_1,S_2),\ldots,(S_{4n-1},S_{4n})$.  Choose pairwise disjoint arcs between the spheres in each pair and take ordinary untwisted internal connected sums in disjoint four-balls.  This gives
\begin{equation}\label{eq:Cj}
 C_j=S_{2j-1}\#S_{2j},\qquad [C_j]=[S_{2j-1}]+[S_{2j}],\qquad C_j^2=-4.
\end{equation}
The two steps are shown in \cref{fig:double-tube}.

\begin{theorem}\label{thm:smooth-minus-four}
The marked double $X_\varphi$ contains $2n$ pairwise disjoint smoothly embedded $(-4)$-spheres $C_1,\ldots,C_{2n}$.
\end{theorem}

\begin{figure}[t]
\centering
\resizebox{.98\textwidth}{!}{%
\begin{tikzpicture}[>=Latex]
\filldraw[fill=turquoise!9,draw=brightblue,rounded corners=8pt,very thick] (0,0) rectangle (6.45,2.15);
\node[font=\bfseries] at (3.225,1.82) {marked fiber sum};
\filldraw[fill=white,draw=brightblue,thick] (.45,.38) rectangle (2.82,1.48);
\filldraw[fill=white,draw=brightgreen!70!black,thick] (3.63,.38) rectangle (6.0,1.48);
\draw[section] (.72,.68)..controls(1.45,.68) and (2.05,1.05)..(2.82,1.05);
\draw[section] (3.63,1.05)..controls(4.40,1.05) and (5.0,.68)..(5.73,.68);
\fill[brightorange] (2.82,1.05) circle (.09);
\fill[brightorange] (3.63,1.05) circle (.09);
\draw[sumarrow] (3.02,1.05)--(3.43,1.05);
\node[font=\small] at (1.63,.12) {$Y_L\setminus\nu F$};
\node[font=\small] at (4.82,.12) {$Y_R\setminus\nu F$};
\node[note,brightred] at (3.225,-.38) {$E_a^L\#E_a^R=S_a$, \quad $S_a^2=-2$};
\begin{scope}[xshift=7.7cm]
\filldraw[fill=brightviolet!7,draw=brightviolet,rounded corners=8pt,very thick] (-.55,0) rectangle (3.0,2.15);
\draw[section] (.1,1.05) circle (.5);
\draw[section] (2.35,1.05) circle (.5);
\draw[tube,line width=4pt] (.60,1.05)--(1.85,1.05);
\draw[white,line width=1.4pt] (.60,1.05)--(1.85,1.05);
\node[font=\bfseries\small,brightviolet,align=center] at (1.225,1.76) {internal connected\\sum};
\node[note] at (1.225,.25) {$C_j^2=-4$};
\end{scope}
\end{tikzpicture}}
\caption{Left: corresponding sections glue to a symplectic $(-2)$-sphere. Right: an orange tube joins two such spheres to form a smooth $(-4)$-sphere; by \cref{thm:adjunction}, its class is not symplectic for the standard sum form.}\label{fig:double-tube}
\end{figure}
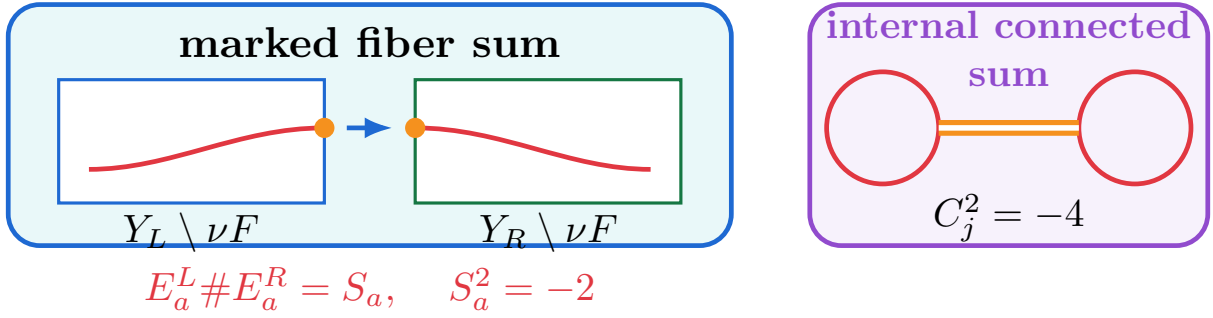

\begin{theorem}\label{thm:adjunction}
For the symplectic sum form of \Cref{prop:minus-two}, none of the classes $[C_j]$ is represented by an embedded symplectic sphere.
\end{theorem}
\begin{proof}
Adjunction for $S_a$ gives $K_{X_\varphi}\cdot S_a=0$.  Thus $K_{X_\varphi}\cdot C_j=0$ while $C_j^2=-4$.  A symplectic sphere in this class would satisfy $-2=C_j^2+K_{X_\varphi}\cdot C_j=-4$, a contradiction.
\end{proof}

\subsection{The fourfold sum}
Take four copies $Y_i=Y(n,k)$.  Use two distinct regular fibers in each middle summand and match corresponding marked section points at all three gluings:
\begin{equation}\label{eq:Z4}
 Z_4(n,k)=Y_1\#_F Y_2\#_F Y_3\#_F Y_4.
\end{equation}

\begin{theorem}\label{thm:symp-minus-four}
The marked sum $Z_4(n,k)$ can be chosen symplectically so that it contains $4n$ pairwise disjoint embedded symplectic spheres $R_1,\ldots,R_{4n}$ with $R_a^2=-4$.
\end{theorem}
\begin{proof}
For each $a$, delete one summing-fiber disk from the section in each end summand and two such disks from the section in each middle summand.  The end pieces are disks and the middle pieces are annuli.  Their boundary-connected chain is a sphere.  Choose product coordinates near every section-fiber intersection in which the fiber and section are the two coordinate factors and the symplectic form is split.  Gompf's fiber-sum construction permits the normal-circle maps to identify these models while preserving the marked section framings \cite{Gompf}.  The section pieces therefore glue symplectically, and their relative Euler numbers add to
\[
 R_a^2=\sum_{i=1}^4(E_a^{(i)})^2=-4.
\]
Distinct labels use disjoint section pieces, so the $R_a$ are pairwise disjoint.
\end{proof}
Figure~\ref{fig:fourfold} displays the chain of four section pieces and the three summing regions.

\begin{corollary}\label{cor:symp-rb}
For every $0\le m\le4n$, symplectically rationally blowing down any $m$ of the spheres $R_a$ produces a symplectic four-manifold $Z_{4,m}(n,k)$.
\end{corollary}
\begin{proof}
Choose pairwise disjoint standard symplectic tubular neighborhoods of the selected spheres.  Symington's theorem applies in each neighborhood, and the supported replacements commute because the neighborhoods are disjoint \cite{Symington}.
\end{proof}

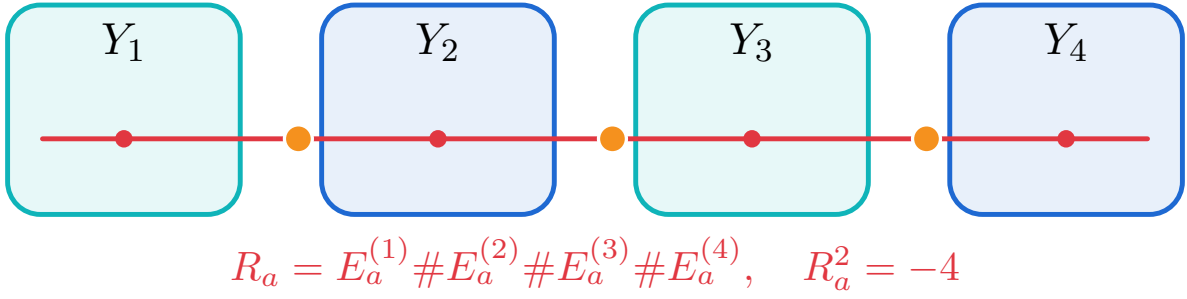
\begin{figure}[t]
\centering
\resizebox{.96\textwidth}{!}{%
\begin{tikzpicture}[>=Latex]
\foreach \x/\lab/\col in {0/$Y_1$/turquoise,2.7/$Y_2$/brightblue,5.4/$Y_3$/turquoise,8.1/$Y_4$/brightblue}{
 \filldraw[fill=\col!10,draw=\col,rounded corners=8pt,very thick] (\x,0) rectangle ++(2.0,1.8);
 \node at (\x+1,1.48) {\lab};
}
\draw[section,line cap=round] (.3,.65)--(9.8,.65);
\foreach \x in {2.35,5.05,7.75}{\filldraw[fill=brightorange,draw=white,line width=.8pt] (\x+.15,.65) circle (.12);}
\foreach \x in {1,3.7,6.4,9.1}{\fill[brightred] (\x,.65) circle (.075);}
\node[note,brightred] at (5.05,-.42) {$R_a=E_a^{(1)}\#E_a^{(2)}\#E_a^{(3)}\#E_a^{(4)},\quad R_a^2=-4$};
\end{tikzpicture}}
\caption{The symplectic fourfold construction.  The end pieces are disks and the two middle pieces are annuli; together they form one symplectic sphere.}\label{fig:fourfold}
\end{figure}

\section{Geometric duals and simultaneous meridians}\label{sec:duals}
Geometric duals kill the meridians that appear when several spheres are removed.
The meridian disk cut from a dual sphere is pictured in \cref{fig:dual}.

\begin{proposition}\label{prop:dual-criterion}
Let $P_1,\ldots,P_r$ be disjoint embedded spheres in an oriented four-manifold $X$.  Suppose that for each $a$ there is an embedded sphere $D_a$ meeting $P_a$ transversely in one positive point and disjoint from every $P_b$ with $b\ne a$; equivalently, the geometric incidence numbers satisfy
\[
 |D_a\cap P_b|=\delta_{ab}.
\]
For every $I\subset\{1,\ldots,r\}$, inclusion induces
\[
 \pi_1\!\left(X\setminus\bigsqcup_{a\in I}\Int\nu(P_a)\right)\cong\pi_1(X).
\]
If the $P_a$ have square $-4$, rational blowdown along any such subcollection also preserves $\pi_1$.
\end{proposition}
\begin{proof}
Removing a normal disk to $P_a$ from $D_a$ leaves a disk in the simultaneous complement whose boundary is a meridian $\mu_a$.  Thus every selected meridian is null-homotopic.  The map from the complement group onto $\pi_1(X)$ has kernel normally generated by these meridians, proving the first assertion.  For rational blowdown, the boundary map
$\pi_1(L(4,1))=\mathbb Z/4\to\pi_1(B_2)=\mathbb Z/2$ is surjective.  Van Kampen, applied one component at a time, then shows that the group is unchanged.
\end{proof}

\begin{figure}[t]
\centering
\resizebox{.82\textwidth}{!}{%
\begin{tikzpicture}[>=Latex]
\filldraw[fill=brightred!6,draw=brightred,very thick] (0,0) circle (1.02);
\filldraw[fill=brightred!3,draw=brightred,dashed,thick] (3.2,0) circle (.76);
\filldraw[fill=brightred!3,draw=brightred,dashed,thick] (5.7,0) circle (.76);
\draw[dual,line cap=round] (-1.45,.48)..controls(-.55,.48) and (.50,.48)..(1.55,1.18);
\filldraw[fill=brightorange,draw=white,line width=.8pt] (-.88,.48) circle (.10);
\node[brightred] at (0,-1.38) {$P_a$};
\node[brightred] at (3.2,-1.1) {$P_b$};
\node[brightred] at (5.7,-1.1) {$P_c$};
\node[brightgreen!70!black] at (1.55,1.28) {$D_a$};
\draw[brightviolet,very thick,->] (-.96,.12) arc (210:500:.34);
\node[brightviolet] at (-1.52,.93) {$\mu_a$};
\node[note] at (3.65,1.45) {$|D_a\cap P_b|=\delta_{ab}$};
\end{tikzpicture}}
\caption{A geometric dual kills the meridian in the simultaneous complement.  The dual spheres need not be mutually disjoint.}\label{fig:dual}
\end{figure}
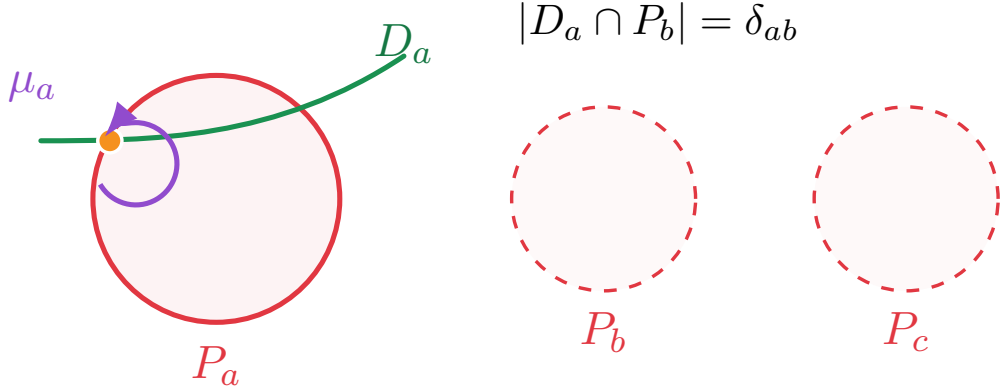

\section{Knot-surgery doubles and branch duals}
Let $K\subset S^3$ be a fibered knot of genus $h\ge1$.  Fintushel and Stern's knot-surgered elliptic surface $E(n)_K$ admits a genus-$2h+n-1$ Lefschetz fibration \cite{FS1998,FS2004}.  Yun subsequently analyzed the monodromy of the associated twisted sums \cite{YunTwisted}.  Only the geometry of this fibration, rather than a particular monodromy word, is used below.

\begin{lemma}\label{lem:resolved-duals}
In the resolved double-cover model of $E(n)$, the $2n$ vertical branch components lift to pairwise disjoint symplectic $(-2)$-sphere sections $Q_a$.  They admit mutually disjoint symplectic $(-2)$-spheres $A_a$ such that $|A_a\cap Q_b|=\delta_{ab}$, with the unique intersection positive and transverse.  The configuration is disjoint from a fixed regular horizontal fiber and from a suitable regular elliptic fiber used for knot surgery.
\end{lemma}
\begin{proof}
Write $Q=\CP^1_x\times\CP^1_y$, and take a branch divisor consisting of four horizontal spheres $H_i=\CP^1_x\times\{q_i\}$ and $2n$ vertical spheres $V_a=\{p_a\}\times\CP^1_y$.  This is the elliptic-surface model, with $8n$ nodes; it is distinct from the two-horizontal-component Gurtas model in the appendix.  Blow up every node $H_i\cap V_a$, and denote the exceptional sphere by $e_{ia}$.  The strict transform $V_a'$ has square $-4$.  In the smooth double cover $\pi\colon\widetilde E(n)\to\widetilde Q$ branched over the disjoint strict transforms, its ramification lift $Q_a=\pi^{-1}(V_a')$ maps with degree one to $V_a'$.  Since $\pi^*[V_a']=2[Q_a]$, the projection formula gives
\[
 2(V_a')^2=(2Q_a)^2,\qquad Q_a^2=-2.
\]
Projection to the $y$-factor makes $Q_a$ a section of the genus-$(n-1)$ horizontal fibration.

The divisor $e_{ia}$ is not a branch component.  The covering restricted to $e_{ia}$ is branched at its two intersections with the transformed branch divisor.  The Riemann-Hurwitz formula identifies its inverse image with a sphere $A_{ia}$, and $A_{ia}^2=2e_{ia}^2=-2$.  We choose the spheres over the nodes on $H_1$ and write $A_a=A_{1a}$.  Since $\pi^*[e_{1a}]=[A_a]$ and $\pi^*[V_b']=2[Q_b]$, the projection formula gives
\[
 2A_a\cdot Q_b=\pi^*[e_{1a}]\cdot\pi^*[V_b']=2(e_{1a}\cdot V_b')=2\delta_{ab}.
\]
Holomorphicity makes the unique intersection positive and transverse.  Exceptional divisors belonging to distinct nodes are disjoint, so the $A_a$ are mutually disjoint.  Hence
\begin{equation}\label{eq:branch-duals}
 |A_a\cap Q_b|=\delta_{ab},\qquad 1\le a,b\le2n.
\end{equation}
The complex double cover supplies a K\"ahler form for which all $Q_a$ and $A_a$ are simultaneously symplectic.  In the resolved model each $A_a$ is contained in the fiber over $q_1$, and hence is disjoint from a fixed regular fiber $F$ over a point outside $\{q_1,\ldots,q_4\}$.  Choose the knot-surgery torus as a regular elliptic fiber of the other projection, over $p_0\notin\{p_1,\ldots,p_{2n}\}$ and away from the resolution neighborhoods.  It is disjoint from every $Q_a$ and $A_a$.  The knot-surgery form agrees with the original form outside this neighborhood, so the untouched configuration remains symplectic.

The local deformation of the four special fibers leaves the chosen regular fiber and the embedded section-dual configuration unchanged.  Thus the $Q_a$ remain sections of the genus-$(2h+n-1)$ fibration, with the $A_a$ retaining the incidences above.  The same persistence argument appears for one section-dual pair in \cite[Lemma~2.2]{AEMS}.
\end{proof}
The branch components, nodes, and the selected resolution spheres are indicated in \cref{fig:branch-duals}.

\begin{proposition}\label{prop:knot-sections}
For $n\ge2$, the fibration on $E(n)_K$ contains $2n$ pairwise disjoint symplectic $(-2)$-sections $Q_a$ with geometric dual spheres $A_a$ satisfying \eqref{eq:branch-duals}.  Moreover, $\pi_1(E(n)_K)=1$.
\end{proposition}
\begin{proof}
Lemma~\ref{lem:resolved-duals} supplies the sections $Q_a$ and their duals $A_a$.  For the fundamental group, use the regular elliptic fiber with simply connected complement in $E(n)$.  Under the Fintushel-Stern boundary identification, the external $S^1$ and a knot meridian map trivially.  The meridian normally generates the knot group, and van Kampen gives $\pi_1(E(n)_K)=1$ \cite{FS1998}.
\end{proof}

Put $g_K=2h+n-1$, the fiber genus of the Lefschetz fibration on $E(n)_K$.  If
$I\subsetneq\{1,\ldots,2n\}$ is nonempty, define
\[
 W_{K,I}=E(n)_K\setminus\Int\nu\!\left(F\cup\bigcup_{a\in I}Q_a\right).
\]
Here $\nu(\cdot)$ denotes a closed regular neighborhood of the indicated normal-crossing union.

\begin{proposition}\label{prop:multiple-complements}
Let $n\ge2$ and let $K$ be a fibered knot of genus $h\ge1$.  For every nonempty proper subset
$I\subsetneq\{1,\ldots,2n\}$, the compact manifold $W_{K,I}$ is simply connected.
\end{proposition}
\begin{proof}
Restoring the regular neighborhood kills the meridians of its components.  As $E(n)_K$ is simply connected, these meridians normally generate $\pi_1(W_{K,I})$.  For $a\in I$, the punctured sphere $A_a\setminus\Int\nu Q_a$ is a meridian disk for $Q_a$ in $W_{K,I}$.  Because $I$ is proper, choose $b\notin I$.  Likewise, $Q_b\setminus\Int\nu F$ is a meridian disk for $F$ and misses the selected sections.  Both types of meridians are null-homotopic, proving $\pi_1(W_{K,I})=1$.
\end{proof}

When $I=\{a\}$, denote the complement by $W_{K,a}=E(n)_K\setminus\Int\nu(F\cup Q_a)$.

\begin{corollary}\label{cor:stein-complement}
The manifold $W_{K,a}$ is a simply connected Stein domain.  Its boundary contact structure is supported by
\[
 (\Sigma_{g_K}^1,t_\delta^2),\qquad g_K=2h+n-1.
\]
For fixed $(n,h)$ its contactomorphism type is independent of $K$ and $a$.
\end{corollary}
\begin{proof}
Apply \cite[Lemma~4.1]{AEMS} to the fiber $F$ and the $(-2)$-section $Q_a$.  It gives a Stein structure on $W_{K,a}$ and identifies the induced contact structure with the one supported by $(\Sigma_{g_K}^1,t_\delta^2)$.  Simple connectivity is \Cref{prop:multiple-complements} with $I=\{a\}$.
\end{proof}
Figure~\ref{fig:stein-complement} summarizes the removal and the boundary open book.

\begin{figure}[H]
\centering
\resizebox{.94\textwidth}{!}{%
\begin{tikzpicture}[>=Latex]
\filldraw[fill=brightblue!8,draw=brightblue,rounded corners=8pt,very thick]
  (0,0) rectangle (3.0,2.05);
\node[font=\bfseries] at (1.5,1.68) {$E(n)_K$};
\draw[fiber] (.45,1.03)--(2.55,1.03);
\foreach \x/\c in {.78/brightred,1.22/brightgreen,1.66/brightviolet,2.10/turquoise}
  \draw[\c,very thick] (\x,.40)--(\x,1.45);
\node[font=\small] at (1.5,.30) {$F,\ Q_a$};
\draw[sumarrow] (3.35,1.03)--(4.75,1.03);
\node[font=\small] at (4.05,1.42) {$\nu(F\cup Q_a)$};
\filldraw[fill=brightgreen!9,draw=brightgreen!70!black,rounded corners=8pt,very thick]
  (5.05,0) rectangle (8.10,2.05);
\node[font=\bfseries] at (6.575,1.48) {$W_{K,a}$};
\node at (6.575,.93) {simply connected};
\node[brightgreen!60!black] at (6.575,.43) {Stein domain};
\draw[sumarrow] (8.45,1.03)--(9.85,1.03);
\filldraw[fill=brightviolet!8,draw=brightviolet,rounded corners=8pt,very thick]
  (10.15,.10) rectangle (13.45,1.95);
\node[font=\bfseries] at (11.80,1.48) {$\partial W_{K,a}$};
\node at (11.80,.92) {$\Sigma_{g_K}^1$ page};
\node at (11.80,.43) {$t_\delta^2$};
\end{tikzpicture}}
\caption{The one-section Stein-complement construction of Akhmedov-Etnyre-Mark-Smith.  The boundary contact structure is supported by the displayed open book.}\label{fig:stein-complement}
\end{figure}
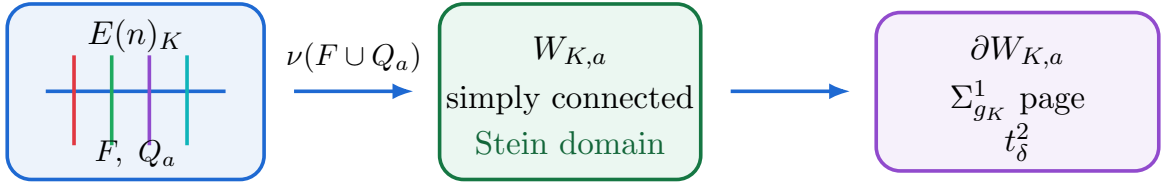

\begin{remark}\label{rem:stein-exotic}
A regular neighborhood of $F\cup Q_a$ is the fiber-section plumbing $X_{g_K,2}$ used in \cite{AEMS}.  For fixed $n$ and $h>1$, varying $K$ among fibered knots of genus $h$ with distinct Alexander polynomials gives infinitely many homeomorphic, pairwise nondiffeomorphic, simply connected Stein fillings of the contact manifold supported by $(\Sigma_{g_K}^1,t_\delta^2)$.  The related construction of Akhmedov-Ozbagci realizes arbitrary finitely presentable fundamental groups \cite{AO}.  For $|I|>1$, \Cref{prop:multiple-complements} establishes only simple connectivity; a Stein analogue would also require allowability and an extension theorem for the larger fiber-section plumbing.
\end{remark}

Let $F$ denote a regular genus-$(2h+n-1)$ fiber and form
\[
 D_K(n)=E(n)_K\#_F E(n)_K
\]
by a gluing that matches the sections of \Cref{prop:knot-sections}.

\begin{theorem}\label{thm:knot-double}
The manifold $D_K(n)$ is simply connected and contains $2n$ pairwise disjoint symplectic spheres
\[
 P_a=Q_a^{(1)}\#Q_a^{(2)},\qquad P_a^2=-4,
\]
with geometric dual spheres $A_a^{(1)}$.  Therefore the simultaneous complement of any subcollection of the $P_a$ is simply connected, and rational blowdown along that subcollection preserves simple connectivity.
\end{theorem}
\begin{proof}
Corresponding punctured sections glue to symplectic spheres of square $-4$.  By \Cref{lem:resolved-duals}, each $A_a^{(1)}$ is disjoint from the fixed regular fiber used in forming $D_K(n)$; it therefore remains a closed sphere after the sum and satisfies
\[
 |A_a^{(1)}\cap P_b|=\delta_{ab}.
\]
Van Kampen presents the complement of a regular fiber as the quotient of the fiber group by the normal closure of the vanishing cycles, together with the fiber meridian.  The punctured section kills that meridian, so each fiber complement is simply connected and $\pi_1(D_K(n))=1$.  Proposition~\ref{prop:dual-criterion} now gives simple connectivity of every simultaneous complement and of each rational blowdown.
\end{proof}

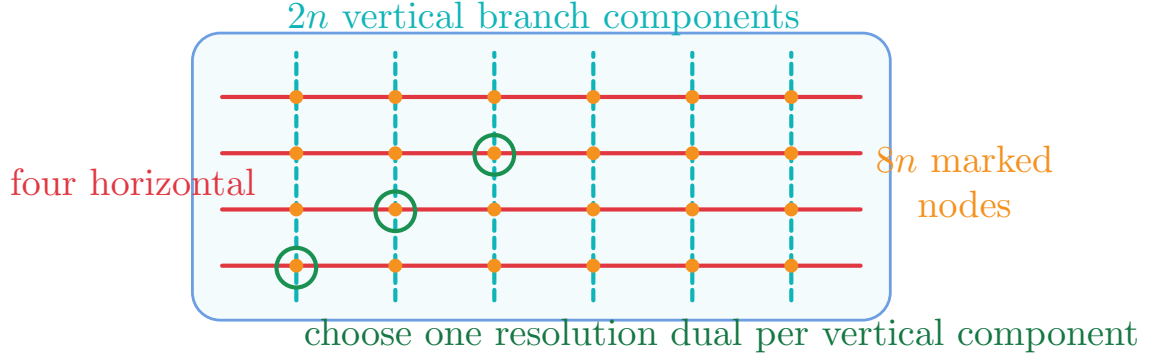
\begin{figure}[t]
\centering
\resizebox{.94\textwidth}{!}{%
\begin{tikzpicture}[>=Latex]
\filldraw[fill=turquoise!5,draw=brightblue!65,rounded corners=8pt,thick] (-.55,-.55) rectangle (6.5,2.35);
\foreach \x in {0,1.25,2.5,3.75}{\draw[section,line cap=round] (-.25,\x/2.2)--(6.2,\x/2.2);}
\foreach \x in {.5,1.5,2.5,3.5,4.5,5.5}{\draw[turquoise,very thick,densely dashed,line cap=round] (\x,-.35)--(\x,2.15);}
\foreach \x in {.5,1.5,2.5,3.5,4.5,5.5}{
 \foreach \y in {0,.568,1.136,1.704}{\fill[brightorange] (\x,\y) circle (.07);}
}
\draw[dual] (.5,-.02) circle (.20);
\draw[dual] (1.5,.55) circle (.20);
\draw[dual] (2.5,1.12) circle (.20);
\node[brightred] at (-1.15,.85) {four horizontal};
\node[turquoise] at (3.0,2.5) {$2n$ vertical branch components};
\node[brightorange,align=center] at (7.25,.85) {$8n$ marked\\nodes};
\node[brightgreen!70!black] at (4.8,-.7) {choose one resolution dual per vertical component};
\end{tikzpicture}}
\caption{Branch-cover origin of the chosen simultaneous dual system.  A vertical branch component lifts to a section; a node-resolution sphere meets that section once and misses the other vertical sections.}\label{fig:branch-duals}
\end{figure}

\section{Invariants, fundamental groups, and blowdown}
From \eqref{eq:Y},
\[
 e(Y)=4-4k+4n,\quad \sigma(Y)=-4n,
 \quad e(F)=4-4k-2n.
\]
Hence
\begin{equation}\label{eq:double-invariants}
 e(X_\varphi)=12n,\quad \sigma(X_\varphi)=-8n,
 \quad \chi_h(X_\varphi)=n,\quad c_1^2(X_\varphi)=0.
\end{equation}
Smoothly rationally blowing down $m\le2n$ of the spheres in \Cref{thm:smooth-minus-four} changes $(e,\sigma)$ by $(-m,+m)$.  No almost-complex structure is asserted for these smooth blowdowns, so $(e+\sigma)/4$ and $2e+3\sigma$ are only numerical abbreviations there.

For the fourfold sum,
\begin{align}\label{eq:fourfold-invariants}
 e(Z_4(n,k))&=8k+28n-8,& \sigma(Z_4(n,k))&=-16n,\\
 \chi_h(Z_4(n,k))&=2k+3n-2,& c_1^2(Z_4(n,k))&=16k+8n-16.
\end{align}
Consequently,
\begin{equation}\label{eq:fourfold-blowdown-invariants}
 \begin{split}
 e(Z_{4,m})&=8k+28n-8-m,\qquad \sigma(Z_{4,m})=-16n+m,\\
 \chi_h(Z_{4,m})&=2k+3n-2,\qquad c_1^2(Z_{4,m})=16k+8n-16+m.
 \end{split}
\end{equation}

For the knot-surgery double, knot surgery preserves $e$ and $\sigma$, so
\begin{equation}\label{eq:knot-invariants}
 e(D_K(n))=28n+8h-8,\qquad \sigma(D_K(n))=-16n,
\end{equation}
and
\[
 \chi_h(D_K(n))=3n+2h-2,\qquad c_1^2(D_K(n))=8n+16h-16.
\]
If $D_{K,m}(n)$ is the symplectic rational blowdown of any $m\le2n$ of the spheres $P_a$ in \Cref{thm:knot-double}, then
\begin{equation}\label{eq:knot-blowdown-invariants}
 \begin{aligned}
 e(D_{K,m}(n))&=28n+8h-8-m, &\qquad \sigma(D_{K,m}(n))&=-16n+m,\\
 b_2^+(D_{K,m}(n))&=6n+4h-5,& b_2^-(D_{K,m}(n))&=22n+4h-5-m.
 \end{aligned}
\end{equation}

Let $V_{-4}$ be the disk bundle of Euler number $-4$ and let $B_2$ be its rational-blowdown replacement.  Then $\partial V_{-4}=L(4,1)$, $\pi_1(B_2)=\mathbb Z/2$, and the boundary homomorphism $\mathbb Z/4\to\mathbb Z/2$ is reduction modulo two.
This replacement is illustrated in \cref{fig:rational-blowdown}.

\begin{proposition}\label{prop:meridian-quotient}
Let $P_1,\ldots,P_m$ be disjoint $(-4)$-spheres in a simply connected four-manifold $D$, let $U=D\setminus\nu(P_1\cup\cdots\cup P_m)$, and let $\mu_i$ be their meridians.  The rational blowdown $D_m$ satisfies
\begin{equation}\label{eq:meridian-quotient}
 \pi_1(D_m)\cong \pi_1(U)/\normal{\mu_1^2,\ldots,\mu_m^2}.
\end{equation}
In particular, $D_m$ is simply connected if the $\mu_i$ are null-homotopic in $U$; more generally it is simply connected exactly when their squares normally generate $\pi_1(U)$.
\end{proposition}
\begin{proof}
Restoring the disk bundles kills the meridians, so they normally generate $\pi_1(U)$.  Gluing $B_2$ instead identifies the generator of $\pi_1(L(4,1))=\mathbb Z/4$ with the generator of $\pi_1(B_2)=\mathbb Z/2$, imposing precisely $\mu_i^2=1$.  Van Kampen gives \eqref{eq:meridian-quotient}.
\end{proof}

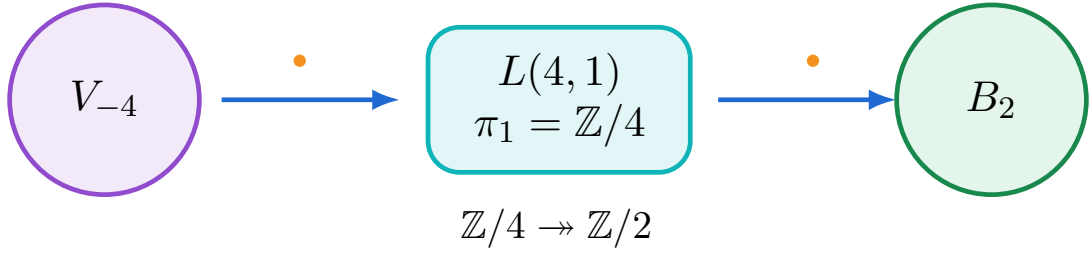
\begin{figure}[t]
\centering
\resizebox{.88\textwidth}{!}{%
\begin{tikzpicture}[>=Latex]
\filldraw[fill=brightviolet!12,draw=brightviolet,very thick] (0,0) circle (.85);
\node at (0,0) {$V_{-4}$};
\draw[sumarrow] (1.05,0)--(2.65,0);
\filldraw[fill=turquoise!12,draw=turquoise,very thick,rounded corners=8pt] (2.9,-.65) rectangle (5.25,.65);
\node[align=center] at (4.08,0) {$L(4,1)$\\$\pi_1=\mathbb Z/4$};
\draw[sumarrow] (5.5,0)--(7.1,0);
\filldraw[fill=brightgreen!12,draw=brightgreen!80!black,very thick] (7.95,0) circle (.85);
\node at (7.95,0) {$B_2$};
\foreach \p in {(1.75,.35),(6.35,.35)}{\fill[brightorange] \p circle (.055);}
\node[note] at (4.05,-1.15) {$\mathbb Z/4\twoheadrightarrow\mathbb Z/2$};
\end{tikzpicture}}
\caption{Rational blowdown replaces the $(-4)$-disk bundle by $B_2$.  If a geometric dual kills the boundary meridian in the exterior, the operation preserves the fundamental group.}\label{fig:rational-blowdown}
\end{figure}

\begin{corollary}\label{cor:exoticity}
Let $n\ge2$, $h\ge1$, and $1\le m\le2n$.  Rationally blow down any $m$ spheres from the dualized collection in \Cref{thm:knot-double}.  The resulting symplectic manifold $D_{K,m}(n)$ is simply connected.  If $m\not\equiv0\pmod8$, then it is homeomorphic but not diffeomorphic to
\[
 \#_{6n+4h-5}\CP^2\ \#\ \#_{22n+4h-5-m}\CPb.
\]
\end{corollary}
\begin{proof}
Simple connectivity follows from \Cref{thm:knot-double,prop:dual-criterion}.  The manifold is closed, oriented, and smooth, so its Kirby-Siebenmann invariant vanishes.  When $m\not\equiv0\pmod8$, the signature $-16n+m$ is not divisible by eight, and the unimodular intersection form is odd.  The displayed values of $b_2^\pm$ make it indefinite.  Freedman's classification identifies the orientation-preserving homeomorphism type with the displayed diagonal connected sum \cite{Freedman}.

By Symington's theorem, $D_{K,m}(n)$ is symplectic.  Since $b_2^+>1$, Taubes' theorem makes its canonical class a Seiberg-Witten basic class \cite{Taubes}.  The diagonal connected sum has vanishing Seiberg-Witten invariants because it splits into two summands with positive $b_2^+$.  Thus the two manifolds are not diffeomorphic.  This parity argument does not decide the case $m\equiv0\pmod8$.
\end{proof}

\begin{proof}[Proof of \Cref{thm:intro-main}]
Combine \Cref{thm:knot-double} with \Cref{cor:exoticity}.
\end{proof}

\section{Concluding remarks}
In a double sum the Gurtas sections yield smooth $(-4)$-spheres, while in a fourfold sum they yield symplectic ones.  For the $2n$ branch sections, the resolution spheres kill the meridians simultaneously; the resulting rational blowdowns are simply connected and are exotic in the range of \Cref{cor:exoticity}.  Removing one branch section together with a regular fiber also gives the Stein complement of \Cref{cor:stein-complement}.  It remains to determine whether the other $2n$ sections in the maximal $4n$-section system admit simultaneous geometric duals.

In follow up work \cite{AkhmedovSakalli}, the authors study Stein fillings obtained by removing suitable fibers and section configurations. These constructions extend the one-section complement perspective to configurations naturally associated with the Lefschetz fibrations considered here.


\appendix
\section{The branch-base and total-space lattices}
The construction uses one intersection lattice on the branch base and another on the total space.  Put $Q_k=\Sigma_k\times S^2$ and let
\[
 A_0=[\Sigma_k\times\{q\}],\qquad B_0=[\{p\}\times S^2],
\]
so $A_0^2=B_0^2=0$ and $A_0\cdot B_0=1$.  Choose distinct points $q_1,q_2\in S^2$ and $p_1,\ldots,p_{2n}\in\Sigma_k$.  The reduced branch divisor is
\[
 \Delta=H_1\cup H_2\cup V_1\cup\cdots\cup V_{2n},
 \quad H_i=\Sigma_k\times\{q_i\},\quad V_j=\{p_j\}\times S^2,
\]
and
\begin{equation}\label{eq:branch-class}
 [\Delta]=2A_0+2nB_0=2(A_0+nB_0).
\end{equation}
Each component occurs with branch multiplicity one.  Each $H_i$ meets each $V_j$ once, giving $4n$ ordinary nodes; see \cref{fig:branch}.

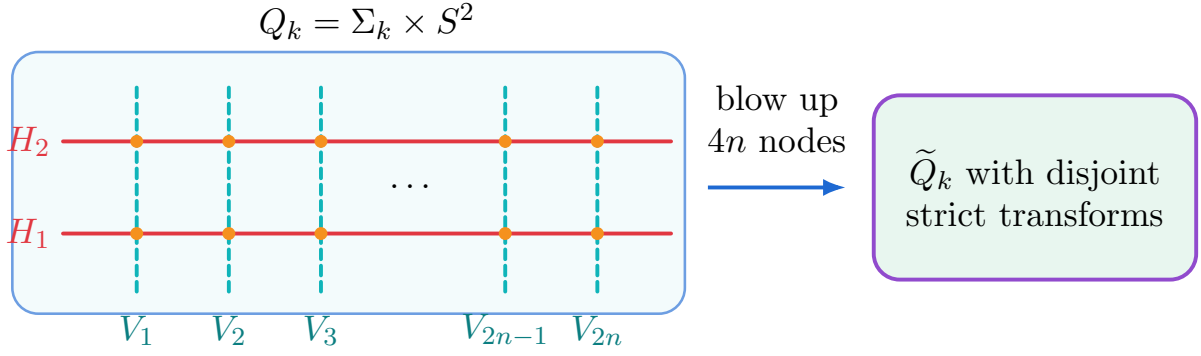
\begin{figure}[t]
\centering
\resizebox{.98\textwidth}{!}{%
\begin{tikzpicture}[>=Latex]
\filldraw[fill=turquoise!5,draw=brightblue!65,rounded corners=8pt,thick] (-.55,-.12) rectangle (6.75,2.72);
\foreach \x/\j in {0.8/1,1.8/2,2.8/3,4.8/{2n-1},5.8/{2n}}{
  \draw[turquoise,very thick,densely dashed,line cap=round] (\x,.12)--(\x,2.38);
  \node[turquoise!70!black] at (\x,-.3) {$V_{\j}$};
}
\node at (3.8,1.25) {$\cdots$};
\draw[section,line cap=round] (0,0.75)--(6.6,.75);
\draw[section,line cap=round] (0,1.75)--(6.6,1.75);
\node[brightred,left] at (0,.75) {$H_1$};
\node[brightred,left] at (0,1.75) {$H_2$};
\foreach \x in {.8,1.8,2.8,4.8,5.8}{\fill[brightorange] (\x,.75) circle (.07);\fill[brightorange] (\x,1.75) circle (.07);}
\node at (3.3,3.05) {$Q_k=\Sigma_k\times S^2$};
\draw[sumarrow] (7.0,1.25)--(8.5,1.25);
\node[align=center] at (7.75,2.0) {blow up\\$4n$ nodes};
\filldraw[fill=brightgreen!10,draw=brightviolet,rounded corners=8pt,very thick] (8.8,.25) rectangle (12.3,2.25);
\node[align=center] at (10.55,1.25) {$\widetilde Q_k$ with disjoint\\strict transforms};
\end{tikzpicture}}
\caption{The branch divisor consists of two reduced horizontal components and $2n$ reduced vertical components.  Their $4n$ transverse intersections are blown up before taking the smooth double cover.}\label{fig:branch}
\end{figure}

Let $\rho\colon\widetilde Q_k\to Q_k$ be the blow-up at these nodes, with exceptional classes $\varepsilon_{ij}$, $i=1,2$, $j=1,\ldots,2n$.  The strict transform satisfies
\begin{equation}\label{eq:strict-transform}
 [\widetilde\Delta]=2A_0+2nB_0-2\sum_{i,j}\varepsilon_{ij}
 =2L,\qquad L=A_0+nB_0-\sum_{i,j}\varepsilon_{ij}.
\end{equation}
The double cover $\pi\colon\widetilde Y\to\widetilde Q_k$ branched along the smooth disjoint divisor $\widetilde\Delta$ resolves the nodal double cover; its identification with $Y(n,k)$ is the one established in the Gurtas-Akhmedov-Saglam construction \cite{Gurtas1,Gurtas2,AS}.  The inverse image of a fiber of the projection $Q_k\to S^2$ is a double cover of $\Sigma_k$ branched at $2n$ points.  Riemann-Hurwitz gives
\[
 2-2g=2(2-2k)-2n,\qquad g=2k+n-1.
\]

The branch-base classes above are not the classes in a ruled blow-up basis of the total space.  Write $A_Y,B_Y,E_1,\ldots,E_{4n}$ for a chosen basis of
$Y(n,k)\cong\Sigma_k\times S^2\#4n\CPb$.  The fiber class recorded by the construction is
\begin{equation}\label{eq:fiber-class}
 F=2A_Y+nB_Y-\sum_{a=1}^{4n}E_a.
\end{equation}
It satisfies $F^2=0$, $E_a\cdot F=1$, and, for
\[
 K_Y=-2A_Y+(2k-2)B_Y+\sum_aE_a,
\]
one has $K_Y\cdot F=4k+2n-4=2g-2$.  Thus $A_0,B_0,\varepsilon_{ij}$ belong to the branch-base lattice, whereas $A_Y,B_Y,E_a$ belong to the total-space lattice.

\section*{Acknowledgments}
A large language model was used only to assist with the preparation and typesetting of the figures.  The authors checked the resulting figures and take full responsibility for the mathematical content of the paper.

\end{document}